%% file: resolution.tex
\documentclass[reqno,a4paper]{amsart}

\usepackage{mathrsfs,amsmath,amssymb,graphicx}
\usepackage{stmaryrd}
\usepackage[unicode]{hyperref}
\usepackage{latexsym}
\usepackage{layout}
\usepackage[english]{babel}
\usepackage{color}
\usepackage{enumitem}
\usepackage{fancyvrb}
\usepackage{color}
\usepackage{amsmath,amscd}

\usepackage{subcaption}

\usepackage{transparent}
\usepackage{xifthen}
\usepackage{tikz}
\usetikzlibrary{calc}
\usetikzlibrary{matrix,arrows,decorations}
\theoremstyle{plain}
\newtheorem{theorem}{Theorem}[section]
\newtheorem{lemma}[theorem]{Lemma}

\theoremstyle{definition}
\newtheorem{definition}{Definition}[section]
\theoremstyle{remark}

\newcommand{\op}[1]{\operatorname{#1}}
\newcommand{\reg}{\op{reg}}
\newcommand{\sing}{{\rm sing}}

\newcommand{\Sta}[1]{{\Gamma_{#1}}}
\newcommand{\pair}{(M,\alpha)}

\newcommand{\Sone}{{\mathbf{S}^1}}
\newcommand{\Stwo}{{\mathbf{S}^2}}

\newcommand{\Sthree}{\mathbf{S}^3}
\newcommand{\Com}{{\mathbb{C}}}
\newcommand{\til}[1]{{\tilde{#1}}}
\newcommand{\ba}[1]{{\hat{#1}}}
\newcommand{\Fib}{\mathsf{S}}

\newcommand{\cout}[1]   {}

\numberwithin{equation}{section}
\numberwithin{figure}{section}

\title{Resolution of singular fibers of an $\Sone$-manifold}
 
\author{Yi-Sheng Wang}
\address{Institute of Mathematics, Academia Sinica, , Taipei City 106, Taiwan}
\email{yisheng@gate.sinica.edu.tw}

\date{\today}

\begin{document}

%{\tiny \tableofcontents}
\thanks{}

\begin{abstract}
In this paper, we present a construction 
of resolution of discrete singular fibers of
a closed $5$-manifold that admits a locally free
$\Sone$-action, and prove its compatibility with 
the resolution of cyclic surface singularities
in the quotient space by the $\Sone$-action. 
\end{abstract}

\maketitle
%%%%%%%%%%%%%%%%%%%%%%%%%%%%%%%%%%%%%%%%%%%%%%%%%%%%%%%%%%%%%%%%%%%%%%%%%%%%%
 
%%%%%%%%%%%%%%%%%%%%%%%%%%%%%%%%%%%%%%%%%%%%%%%%%%%%%%%%%%%%%%%%%%%%%%%%%%%%%
\section{Introduction}\label{sec:intro}
\input{intro}

%%%%%%%%%%%%%%%%%%%%%%%%%%%%%%%%%%%%%%%%%%%%%%%%%%%%%%%%%%%%%%%%%%%%%%%%%%%%%

%%%%%%%%%%%%%%%%%%%%%%%%%%%%%%%%%%%%%%%%%%%%%%%%%%%%%%%%%%%%%%%%%%%%%%%%%%%%%
\section{Preliminaries}\label{sec:prelim}
\input{prelim}

%%%%%%%%%%%%%%%%%%%%%%%%%%%%%%%%%%%%%%%%%%%%%%%%%%%%%%%%%%%%%%%%%%%%%%%%%%%%%%

%%%%%%%%%%%%%%%%%%%%%%%%%%%%%%%%%%%%%%%%%%%%%%%%%%%%%%%%%%%%%%%%%%%%%%%%%%%%%%
\section{Resolution of singularties via charts}\label{sec:resolution_charts}

\input{charts}

%\section{Resolution of singularties via embeddings}\label{sec:resolution_embeddings}
%\input{resol.embeddings}
%%%%%%%%%%%%%%%%%%%%%%%%%%%%%%%%%%%%%%%%%%%%%%%%%%%%%%%%%%%%%%%%%%%%%%%%%%%%%%

\section*{Acknowledgment}
The author would like to thank Prof. Shu-Cheng Chang
for introducing him to the topic and helpful discussions
during the preparation of the manuscript.

%%%%%%%%%%%%%%%%%%%%%%%%%%%%%%%%%%%%%%%%%%%%%%%%%%%%%%%%%%%%%%%%%%%%%%%%%%%%%

\end{document}

%% file: intro.tex
$\Sone$-actions on odd-dimensional manifolds
naturally arise in the study of contact manifolds. 
In particular, any almost regular contact manifold
admits a locally free $\Sone$-action \cite{Tho:76}, \cite{Wad:75}. Recent research progress in the field of Sasakian
geometry further fosters the interest in odd-dimensional manifolds that admit an $\Sone$-action. Investigation along this line has been 
carried out, for instance, in \cite{Kol:05}, \cite{Kol:06},
and resulted in greater insight into the topology of Sasakian manifolds. 
As the structure of Sasakian $3$-manifolds
is now well-understood and classified \cite{Bel:01}, 
attention has been drawn to the classification 
of Sasakian $5$-manifolds. For a compact $5$-manifold, it is known that 
$M$ admits a Sasakian structure
if and only if it admits a quasiregular Sasakian
structure \cite{Ruk:95}. In the latter case, it 
admits a natural locally free $\Sone$-action, and
the quotient $M/\Sone$ 
has a K\"ahler orbifold structure (see \cite[Theorem $7.1.3$]{BoyGal:07})
with points in the singular locus 
corresponding to the fibers with non-trivial stabilizers in $M$.

Motivated by this, we consider
a closed $5$-manifolds $M$ that 
admits a locally free, effective $\Sone$-action $\alpha$
such that each orbit admits an $\Sone$-invariant complex
neighborhood (see paragraphs preceding Definition \ref{def:special_complex_S1_manifolds}), 
called a complex 
$\Sone$-manifold $(M,\alpha)$ hereafter; 
the quotient $M/\Sone$
is a complex surface with singularities.
The paper presents a construction of resolution of 
discrete singular fibers in $(M,\alpha)$, 
and shows that it is compatible with
the resolution for surface cyclic quotient
singularties in $M/\Sone$ 
given in \cite{Rei} (see also \cite{Rei:74}).

%maybe not the following one
%Jung, H. W. E. (1908), "Dars%tellung der Funktionen eines algebraischen Körpers zweier unabhängigen Veränderlichen x,y in der Umgebung x=a, y= b", Journal für die Reine und Angewandte Mathematik, 133: 289–314

%Pages 335-364 from Volume 122 (1985), Issue 2 by Ronald Fintushel, Ronald J. Stern
%introduction
%definition of Seifert manifold
%complex/action
%%%%%%%%%%%%%%%%%%%%%%%%%%%%%%%%%%%%%%%5
%reference Jung's paper Reid's paper and  
%Riemenschneider's paper 
%think about how Jung resolove codimenional one singularites

\begin{theorem}\label{teo:resolution}
Given a complex $\Sone$-manifold $(M,\alpha)$
with discrete singular fibers $\Fib_i$, $i=1,\dots,n$, 
there exists an $\Sone$-manifold $(M_\omega,\alpha_\omega)$ 
and disjoint 
$\Sone$-invariant $3$-dimensional subspace $P_{i}$ 
of $M_\omega$, $i=1,\dots, n$, such that
\begin{enumerate}%[label=\roman(*)]
\item $\alpha_\omega$ is free; %\label{itm:free} 
\item $M_\omega-\bigcup_{i=1}^n P_i$  
is $\Sone$-equivariantly diffeomorphic to $M-\bigcup_{i=1}^n \Fib_i$; %\label{itm:isomorphic}
\item $M_\omega/\Sone\rightarrow M/\Sone$ is the resolution of
$M/\Sone$ as a singular complex surface;%\label{itm:compatible}
\item $P_i$ is the union $Q_{i1}\cup\cdots\cup Q_{ic(i)}$
with $Q_{ij}$  diffeomorphic to $\Stwo\times \Sone$,
$1\leq j<c(i)$, and $Q_{c(i)}$ diffeomorphic to $\Sthree$, and 
$Q_{ij}\cap Q_{ij'}$ an $\Sone$-invariant circle
in $Q_{ij}, Q_{ij'}$ when $j'= j+1$ and empty otherwise, %\label{itm:chain}
\end{enumerate}
where $c(i)$ is the length of the singular fiber $\Fib_i$ (see Definition \ref{def:complexity_singular_fiber}).
\end{theorem}

We remark that the subspace $P_i$
is a reminiscence to the chain of $2$-spheres
in the plumbing construction in $4$-dimensions.

Notation and basic facts 
about $\Sone$-manifolds
are reviewed in Section \ref{sec:prelim},
and the proof of Theorem \ref{teo:resolution}
occupies Sections \ref{subsec:reduction}, \ref{subsec:topology_Pi}
and \ref{subsec:compatibility}.
%Section \ref{subsec:comparison}
%compares the resolution presented here
%and the one, via embedding, given in
%\cite{} for $\frac{1}{r}(1,1)$ type singular fibers. 

%% file: prelim.tex
 
\subsection{$\Sone$-manifold}
Throughout the paper, $\Sone$
is the unit circle in the complex plane, 
%be an oriented circle, 
and $(M,\alpha)$ denotes a closed smooth $5$-manifold $M$ 
%denote by $(M,\alpha)$
%an oriented closed smooth $5$-manifolds $M$ 
equipped with a locally free, effective
$\Sone$-action $\alpha:\Sone\times M\rightarrow M$. 
%such that $\alpha(t)$ is orientation-preserving, for every 
%$t\in \Sone$. 
Since $M$ is compact, $\alpha$ is proper and the stabilizer
of every point $p\in M$ is finite (i.e.\ almost free).
In particular, by the slice theorem \cite[Theorem $3.8$]{Pal:60} 
(see \cite[Theorem $5.6$]{Die:87}, \cite{Bre:72} and \cite{Dav:08}),
the orbit space $O:=M/\Sone$ is an orbifold \cite[Theorem $1.5$]{Car:19}.
Denote by $\pi$ the quotient map $M\rightarrow O$,
and call $(M,\alpha)$ an $\Sone$-manifold. 
%We remark that $\alpha$  
%induces a foliation $\mathfrak{F}_\alpha$.

Since $\alpha$ is effective, the principal orbit type
$M_{\reg}$ corresponds to points in $M$
with the trivial stabilizer. By the principal orbit type 
theorem \cite[Theorem $5.14$]{Die:87} (see also \cite[Theorem $1.32$]{Mei:03}), it is open, dense and connected; in addition, 
$M_{\reg}\rightarrow M_{\reg}/\Sone$
is a principle $\Sone$-bundle. 
\begin{definition}
The orbit $\Fib$ of a point $p\in M$ 
is called a regular fiber 
if $p\in M_{\reg}$; otherwise it is called a singular fiber.
\end{definition}

In general, $M-M_{\reg}$ is a union
of smooth submanifolds of $M$. 
The present note concerns mainly the case where
$M-M_{\reg}$ is a discrete, and hence finite, set.

Two $\Sone$-manifolds $(M,\alpha),(M',\alpha')$
are said to be $\Sone$-equivalent if there exists
a diffeomorphism between $M,M'$ that respects $\alpha,\alpha'$.
By the slice theorem,
each orbit $\Fib$ of $p$ 
in an $\Sone$-manifold $(M,\alpha)$ 
admits an $\Sone$-invariant neighborhood $U_{\Fib}$
which is $\Sone$-equivalent to
a model $\Sone$-manifold  
$V_p:=\Sone\times \mathbb{R}^4 / \rho_p$,
where $\rho_p$ is a faithful representation of 
the isotropy subgroup
$\Sta{p}$ of $p$ 
on the real vector space $\mathbb{R}^4$,
and the $\Sone$-action on $V_p$ is induced by
the $\Sone$-action on $\Sone$.
Note that   
$\rho_p$ is faithful 
since $\alpha$ is effective. 

The present paper 
concerns
%considers
the case where 
$\mathbb{R}^4$ is regarded as the complex vector space 
$\Com^2$ and $\rho_p$ is complex linear; we call such a
%$V$ is endowed with a complex structure,
%and 
$V_p:=\Sone\times \Com^2 / \rho_p$ 
a model complex $\Sone$-manifold.
%$V_p:=\Sone\times \Com^2 / \rho_p$ 

%to be a model $\Sone$-manifold
%with $\rho_p$ complex linear. 
%This leads us to the following definition.

\begin{definition}\label{def:special_complex_S1_manifolds}
A special complex $\Sone$-manifold is an $\Sone$-manifold 
with $M_\sing:=M-M_{\op{reg}}$ discrete, 
and for every orbit $\Fib$,
there exists an $\Sone$-invariant neighborhood $U_{\Fib}$ 
which is $\Sone$-equivalent to
a model complex $\Sone$-manifold. 
%$V_p:=S^1\times V / \alpha_p$.
\end{definition}

%Note that the stabilizer $\Sta{p}$ of a point 
%is isomorphic to a finite cyclic group $\mathbb{Z}_r$,
%for some $r\in\mathbb{N}$.
%This allows to define the type of a singular fiber.
\begin{definition}\label{def:types_singular_fiber}
Given an orbit $\Fib$ of $p\in M_\sing$, 
we say it is a singular fiber of type $\frac{1}{r}(1,a)$, 
$r,a$ two positive coprime integers,  
if the stabilizer $\Sta{p}$ is isomorphic to $\mathbb{Z}_r=\langle g\rangle$
such that $\rho_p$ is conjugate to the representation 
\begin{align*}
\mathbb{Z}_r&\rightarrow \Com^2\\
g&\mapsto 
\begin{bmatrix}
e^{\frac{2\pi i}{r}}& 0\\
0& e^{2\pi i\frac{a}{r}}
\end{bmatrix}.
\end{align*}
\end{definition}

Up to conjugation,
a faithful representation $\rho$ 
of $\mathbb{Z}_r$ on $\Com^2$
is conjugate to 
\begin{align*}
\mathbb{Z}_r&\rightarrow \Com^2\\
g&\mapsto 
\begin{bmatrix}
e^{2\pi i\frac{a_1}{r_1}}& 0\\
0& e^{2\pi i\frac{a_2}{r_2}},
\end{bmatrix}
\end{align*}
for some positive integers $r_1,r_2$ with their least common multipler $r$, and $a_i, i=1,2$ integers coprime to $r_i, i=1,2$.

If $r_1,r_2<r$, then $(M,\alpha)$
has non-discrete set of singular fibers, that is, 
$M_\sing$ contains some $3$-dimensional submanifolds of $M$.
On the other hand, $M_\sing$ is discrete if and only if $r_1=r_2=r$. 
Up to a change of generator,
the representation $\rho$ can be normalized as follows:  
\begin{align*}
\mathbb{Z}_r&\rightarrow \Com^2\\
g&\mapsto 
\begin{bmatrix}
e^{2\pi i\frac{1}{r}}& 0\\
0& e^{2\pi i\frac{a}{r}},
\end{bmatrix}
\end{align*}
for some $a>0$ coprime to $r$.
Therefore, we have the following.
\begin{lemma}
Every discrete singular fiber is of type $\frac{1}{r}(1,a)$,
for some coprime postiive integers $r, a$.
\end{lemma} 
 
From now on, $\pair$ is assumed to be a speical complex
$\Sone$-manifold. 
\subsection{Neighborhood of a singularity}\label{subsec:neighborhoods}
%Let $\pair$ be a special complex $\Sone$-manifold.
In this subsection, we examine the $\Sone$-structure 
of a neighborhood of a singular fiber of $\frac{1}{r}(1,a)$ type.
Let $\Fib$ be a singular fiber, and $U_p$ 
an $\Sone$-invariant neighborhood of 
$\Fib$ which is $\Sone$-equivalent to a model complex $\Sone$-manifold
$\Sone\times \Com^2/\mathbb{Z}_r$,
where $\mathbb{Z}_r$, generated by $g$, acts on $\Com^2$
by the representation 
\begin{align*}
\rho: \mathbb{Z}_r&\rightarrow \Com^2\\
g&\mapsto 
\begin{bmatrix}
e^{2\pi i\frac{1}{r}}& 0\\
0& e^{2\pi i\frac{a}{r}},
\end{bmatrix}
\end{align*}
and the action of $\mathbb{Z}_r$ on $\Sone$
is given by letting $g=e^{2\pi i\frac{1}{r}}$.
%$\mathbb{Z}_r=<t>$ as the subset 
%\[\{t^n=e^{2\pi i\frac{-n}{r}}\mid n=0,\dots,r-1\}\]
%of $S^1$ and the Lie group structure of $S^1$. 
%
The $\Sone$-action on $\Sone\times \Com^2/\mathbb{Z}_r$
is induced by the $\Sone$-action on the first factor of 
$\Sone\times \Com^2$.
%given by 
%\begin{align*}
%\Sone\times \Sone\times \Com^2&\rightarrow \Sone\times \Com^2\\
%(t,w,z_1,z_2)&\mapsto (tw,z_1,z_2).
%\end{align*} 

Consider the local diffeomorphism 
\begin{align*}
\pi:\Sone\times \Com^2 &\rightarrow \Sone\times \Com^2\\
(w,z_1,z_2)&\mapsto (w^r,w^{-1}z_1, w^{-a} z_2) 
\end{align*}
which descends to a diffeomorphism from 
$\Sone\times \Com^2/\mathbb{Z}_r$
to $\Sone\times \Com^2$, and therefore,  
the neighborhood $U_\Fib$ is $\Sone$-equivalent to 
the $\Sone$-manifold $(\Sone\times \Com^2,\alpha_0)$
with $\alpha_0$ given by   
\begin{align}\label{eq:s1_action_nbhd_singular}
\Sone\times \Sone\times \Com^2 &\rightarrow \Sone\times \Com^2\\
(t,u,v_1,v_2)&\mapsto (t^r u,t^{-1}v_1, t^{-a} v_2)\nonumber.
\end{align}

%% file: charts.tex
\subsection{Reduction on type $\frac{1}{r}(1,a)$ singular fibers}\label{subsec:reduction}
Let $\Fib$ be a singular fiber of $\pair$
$\Fib$. 
We measure its complexity by continued fractions as follows.
\begin{definition}
Suppose $\Fib$ is of type $\frac{1}{r}(1,a)$,
and $[b_1, \dots, b_n]$ is the continued fraction of $\frac{r}{a}$, namely 
\[\frac{r}{a}=b_1-\frac{1}{b_2-\frac{1}{b_3-\dots}}.\]
Then we define the length of the singular fiber to be $n$. 
\end{definition}\label{def:complexity_singular_fiber}
Note that the fraction can be calculated by 
the recursive formula:
\begin{align}
r&=a b_1-a_1, 0<a_1<a\nonumber\\
a&=a_1 b_2-a_2, 0<a_2<a_1\nonumber\\
a_i&=a_{i+1} b_{i+2}-a_{i+2}, 0<a_{i+2}<a_{i+1}\nonumber\\
&\dots \label{eq:bi_ai}.
\end{align}

Now we present a construction that reduces 
a type $\frac{1}{r}(1,a)$ singular fiber
to a type $\frac{1}{a}(1,a_1)$ singular fiber,
where
$\frac{a}{a_1}=[b_2,\dots, b_n]$ and
$\frac{r}{a}=[b_1,b_2,\dots, b_n]$.

Consider the five manifold $\Com\times \Sthree$ 
equipped with the $\Sone$-action 
\begin{align*}
\delta_a:\Sone\times (\Com\times \Sthree) &\rightarrow \Com\times \Sthree\\
(x,y_1,y_2)&\mapsto (t^r x,t^{-1}y_1,t^{-a}y_2),
\end{align*}
where $(x,y_1,y_2)\in \Com\times \Sthree\subset\Com^3$.

\begin{lemma}
$(\Com\times \Sthree,\delta_a)$ 
has only one singular fiber and it is of type $\frac{a_1}{a}$. 
\end{lemma}
\begin{proof}
Observe first that the manifold $\Com\times \Sthree$
can be obtained by gluing manifolds
$X:=\Com\times\Sone \times \Com $
and $X':=\Com\times\Com\times \Sone$
via the diffeomorphism
\begin{align}
X'\supset\Com\times \Com^\ast \times \Sone &\xrightarrow{f} \Com\times \Sone \times \Com^\ast\subset X\nonumber\\
(p',q_1',q_2')&\mapsto (p'\vert q_1'\vert,\frac{q_1'}{\vert q_1'\vert},\frac{q_2'}{\vert q_1'\vert}),\label{eq:gluing_map_XprimeX}
\end{align}
where $\Com^*=\Com-\{0\}$.
The manifold $\Com\times\Sthree$ can be identified with $X  \cup_f X'$ via the embeddings:
\begin{align}
\iota:X &\rightarrow \Com\times \Sthree\nonumber\\
(p,q_1,q_2)&\mapsto (p\vert q \vert,\frac{q_1}{\vert q\vert},
\frac{q_2}{\vert q\vert}),\label{eq:identification_X_ComSthree}\\
\iota':X'&\rightarrow \Com\times \Sthree\nonumber\\
(p',q_1',q_2')&\mapsto (p'\vert q'\vert,\frac{q_1'}{\vert q'\vert},\frac{q_2'}{\vert q'\vert}),\label{eq:identification_Xprime_ComSthree}
\end{align}
where $q=(q_1,q_2)$ and $q'=(q_1', q_2')$.
It is not difficult to check the embeddings
respect $f$, and hence induces a diffeomorphism 
between $X\cup_f X'$ and $\Com\times\Sthree$.

The $\Sone$-action $\delta_a$ on $\Com\times\Sthree$ 
restricts to $\Sone$-actions on $X,X'$ via \eqref{eq:identification_X_ComSthree}, \eqref{eq:identification_Xprime_ComSthree} as follows:
\begin{align*}
\kappa_a :\Sone \times X &\rightarrow X \nonumber\\
(t, p, q_1, q_2)&\mapsto (t^r p,t^{-1}q_1,t^{-a}q_2)%\label{eq:Sone_action_charts1}
\\
\kappa_a':\Sone \times X'&\rightarrow X'\nonumber\\
(t, p', q_1', q_2')&\mapsto (t^r p',t^{-1}q_1',t^{-a}q_2').
%\label{eq:Sone_action_charts2} 
\end{align*}
This shows $\Fib_0=\{(0,0)\}\times\Sone$
is a singular fiber of $(X',\kappa_a')$;
on the other hand, the $\Sone$-action $\kappa$ on $X$ 
is free since $q_1$ is never zero.

To determine the type of the singular fiber $\Fib_0$, consider the representation of $\mathbb{Z}_a=\langle g\rangle$ 
on $\Com^2$:
\begin{align*}
\rho: \mathbb{Z}_a&\rightarrow \Com^2\\
g&\mapsto 
\begin{bmatrix}
e^{2\pi i\frac{a_1}{a}}& 0\\
0& e^{2\pi i\frac{1}{a}},
\end{bmatrix}
\end{align*}
and the natural homomorphism 
\begin{align*}
\iota: \mathbb{Z}_a &\rightarrow \Sone\\
g&\mapsto e^{2\pi i\frac{1}{a}}. 
\end{align*} 
Equip $\til X=\Com^2\times \Sone$ with the $\Sone$-action
acting on the second factor.
Then the quotient $\Com^2\times\Sone/\mathbb{Z}_a$ given 
by $\rho,\iota$ is $\Sone$-equivalent to $(X',\kappa_a')$
via the map 
\begin{align}
\til X'=\Com^2\times \Sone &\rightarrow X'\nonumber\\
(\til p',\til q_1',\til q_2')&\mapsto 
((\til q_2')^{r}\til p',(\til q_2')^{-1}\til q_1',(\til q_2')^{-a})\label{eq:tilXprime_to_Xprime}.
\end{align} 
Note that the map sends the orbit of the $\mathbb{Z}_a$-action
to a point because $r=ab_1-a_1$. Thus by Definition \ref{def:types_singular_fiber},
the fiber is of type $\frac{1}{a}(1,a_1)$.
\end{proof}

Suppose $\Fib$ a singular fiber of $(M,\alpha)$ of type 
$\frac{1}{r}(1,a)$ with length $n$. 
Then by \eqref{eq:s1_action_nbhd_singular}, 
we can identify an $\Sone$-invariant neighborhood 
$U_\Fib$ of $\Fib$ with $\Sone\times\Com^2$ 
such that $\Fib=\Sone\times \{(0,0)\}$
and the $\Sone$-action on $U_\Fib$ is realized as follows:
\begin{align}
\Sone\times U_\Fib &\rightarrow U_\Fib\nonumber\\
(t,u,v_1,v_2)&\mapsto (t^ru,t^{-1}v_1,t^{-a}v_2).\label{eq:model_neighborhood_sing_fib}
\end{align} 

Consider the $\Sone$-equivariant embedding
\begin{align}
e:(U_\mathsf{S}-\mathsf{S})&\rightarrow
\Com\times \Sthree\nonumber\\
(u,v_1,v_2)&\mapsto (u\vert v\vert, \frac{v_1}{\vert v\vert},\frac{v_2}{\vert v\vert}),\label{eq:gluing_embedding}
\end{align}
where $v=(v_1,v_2)$, and glue 
the manifold $(M-\Fib,\alpha)$ and $(\Com\times\Sthree,\delta_a)$
via $e$. The resulting new special complex $\Sone$-manifold $(M',\alpha')$
has all singular fibers the same as $\pair$ except that 
the singular fiber $\mathsf{S}$ is now replaced with a (singular) fiber of smaller length. By induction,
we eventually get a special complex $\Sone$-manifold
$(M_\omega,\alpha_\omega)$ with $\alpha_\omega$ free, and
this proves the first assertion of Theorem \ref{teo:resolution}.

We remark that $e$ can be decomposed into the embeddings \eqref{eq:identification_X_ComSthree} and \eqref{eq:identification_Xprime_ComSthree}
in terms of $X,X'$ as follows: 
\begin{align}
U_\Fib\supset\Sone\times\Com^\ast \times\Com  &\rightarrow X\nonumber\\
(u,v_1,v_2)&\mapsto (u\vert v_1\vert,\frac{v_1}{\vert v_1\vert},\frac{v_2}{\vert v_1\vert})\label{eq:gluing_embedding1_5_dim}\\
U_\Fib\supset\Sone\times\Com \times\Com^\ast  &\rightarrow X'\nonumber\\
(u,v_1,v_2)&\mapsto (u\vert v_2\vert,\frac{v_1}{\vert v_2\vert},\frac{v_2}{\vert v_2\vert})\label{eq:gluing_embedding2_5_dim}.
\end{align}
\eqref{eq:gluing_embedding1_5_dim}, \eqref{eq:gluing_embedding2_5_dim} come in handy
when proving the rest assertions in Theorem \ref{teo:resolution}.

\subsection{The topology of the $3$-dimensional subspace $P_i$}\label{subsec:topology_Pi}
To understand the structure of $P_i$ in 
Theorem \ref{teo:resolution}, 
we note that without loss of generality, 
it may be assumed that $\pair$
has only one singular fiber $\Fib$. 
In particular, if $a=1$, 
$\alpha_\omega=\alpha'$, and in $M_\omega$, 
$U_\Fib$ is replaced with $X,X'$
and the singular fiber $\Fib$
replaced with the $3$-sphere, 
$P:=R\cup_f R'$, where
$R=0\times \Sone\times\Com, R'=0\times\Com\times\Sone$  
and $f$ is given in \eqref{eq:gluing_map_XprimeX}.
For the general case, we denote by 
$(X_i,\kappa_{a_i})$,  
$(X_i',\kappa_{a_i}')$ 
the $\Sone$-manifolds 
$\Com\times\Sone\times\Com$, $\Com\times\Com\times\Sone $
equipped with the 
$\Sone$-actions
\begin{align*}
\kappa_{a_i} :\Sone \times X_i &\rightarrow X_i \nonumber\\
(t, p, q_1, q_2)&\mapsto (t^{a_{i-1}} p',t^{-1}q_1',t^{-a_i})
%\label{eq:Sone_action_charts1}
\\
\kappa_{a_i}':\Sone \times X_i'&\rightarrow X_i'\nonumber\\
(t, p', q_1', q_2')&\mapsto (t^{a_{i-1}} p',t^{-1}q_1',t^{-a_i}q_2'),
%\label{eq:Sone_action_charts2} 
\end{align*}
respectively, $i=0,\dots,n-1$,
with $n$ being the length of the singular fiber $\Fib$
and $a_{-1}=r,a_0=a$ and $a_i,i>0$ given by \eqref{eq:bi_ai}.
Then we have the following lemma.
\begin{lemma}\label{lm:chain_charts}
%Let $(X_i,\kappa_{a_i})$ 
%be a copy of the $\Sone$-manifold $(\Com\times\Sone\times\Com,\kappa_{a_i})$ in \eqref{eq:Sone_action_charts1} with 
%$a_{-1}=r,a_0=a$ and $a_i,i>0$ given in \eqref{eq:bi_ai}, 
%and 
Let $R_i, R_i^\perp$ be
the $\Sone$-invariant subspaces $0\times\Sone\times\Com, 
\Com\times\Sone\times 0\subset X_i$, respectively, 
$i=0,\dots, n-1$,
%let $(X_{n-1}',\kappa_1')$ be a copy of $(\Com\times \Com\times\Sone,\kappa_1')$ in \eqref{eq:Sone_action_charts2},
and $R_{n-1}'$ the $\Sone$-invariant subspace 
$0\times\Com\times\Sone\subset X_{n-1}'$. 

Then there exist $\Sone$-equivariant diffeomorphisms 
$g_i$, $i=1,\dots, n-1$, from 
$\Com\times\Sone\times\Com^\ast\subset X_{i-1}$
to $\Com^\ast\times\Sone\times \Com \subset X_i$ sending
$(0,\lambda,\mu)\in 0\times \Sone\times \Com^\ast\subset R_{i-1}$ to $(\mu^{-1},\lambda,0)\in R_{i}^\perp$ such that
the resolution $(M_\omega,\alpha_\omega)$
is obtained by 
gluing $U_\Fib-\Fib$ to
\[X_\omega:=X_0\cup_{g_1} X_1\cup_{g_2} X_2\cup\cdots\cup_{g_{n-1}} X_{n-1}\cup_{g} X_{n-1}'\]
via an $\Sone$-equivariant embedding $e_\omega$
that restricts to 
\[U_\Fib-(\Sone\times 0\times\Com)=\Sone\times \Com^\ast\times \Com\rightarrow X_0\] 
given in
\eqref{eq:gluing_embedding1_5_dim}, 
and $g$ is the diffeomorphism in \eqref{eq:gluing1_topological_model}.
Furthermore, $U_\Fib-\Fib$
is $\Sone$-equivalent to $X_\omega-R_\omega$,
with  
\[R_\omega:=R_0\cup_{g_1} (R_1^\perp\cup R_1)\cup_{g_2} \cdots\cup_{g_{n-1}} (R_{n-1}^\perp\cup R_{n-1})\cup_{g} R_{n-1}'.\]
\end{lemma}
Lemma \ref{lm:chain_charts} implies
the second and forth assertions of Theorem \ref{teo:resolution} where $P$ corresponds to $R_\omega$ here, and 
$Q_j:=R_{j-1}\cup_{g_j} R_{j}^\perp$, $j=1,\dots,n-1$,
and $Q_n:=R_{n-1}\cup_g R_{n-1}'$. 
It is then not difficult to check 
$Q_{n}$ is diffeomorphic to a $3$-sphere, 
$Q_j$ diffeomorphic to $\Stwo\times\Sone$, and 
$Q_{j}\cap Q_{j+1}$ the circle $R_j\cap R_j^\perp=0\times\Sone\times 0$,
$j=1,\cdots, n-1$.

\begin{proof}
We prove by induction. 
The case $n=1$ is clear, so we may assume $a>1$,
and we apply the resolution construction to 
the singular fiber $\Fib_0$ 
in $(M',\alpha')$. To do so, 
we observe first that 
$X'$ is an $\Sone$-invariant neighborhood
of $\Fib_0$, but the $\Sone$-action is not quite 
the same as $U_S$ in form, so we 
``normalize'' $X'$ 
by the $\Sone$-equivariant diffeomorphism
\begin{align}
X'&\rightarrow X''=\Sone\times\Com^2 \label{eq:normalize}\\
(p',q_1',q_2')&\mapsto ((q_2')^{-1},q_1',p'(q_2)^{b_1}).\nonumber
\end{align}
The induced $\Sone$-action on $X''$ is then given by 
\begin{align}
\Sone\times X''&\rightarrow X'' \\
(p'',q_1'',q_2'')&\mapsto (t^a p'',t^{-1}q_1'', t^{-a_1}q_2''). \nonumber
\end{align}

With the new coordinates, we can now 
apply the resolution construction to $X''$.
Take a copy of $\Com\times\Sthree$
and identify it with $X_1\cup_{f_{1}} X_1'$
as in \eqref{eq:identification_X_ComSthree},
\eqref{eq:identification_Xprime_ComSthree},
where $X_1,X_1'$
are copies of $X,X'$, respectively, and $f_1=f$.
Think of $X''$ as the union of $\Sone\times\Com^\ast\times\Com\cup\Sone\times\Com\times\Com^\ast$.
Then as in \eqref{eq:gluing_embedding1_5_dim} and \eqref{eq:gluing_embedding2_5_dim}, we have the gluing maps:  
\begin{align}
X''\supset\Sone\times\Com^\ast \times\Com  &\rightarrow X_1 \nonumber\\
(u,v_1,v_2)&\mapsto (u\vert v_1\vert,\frac{v_1}{\vert v_1\vert},\frac{v_2}{\vert v_1\vert})\nonumber
%\label{eq:gluing_embedding_Xprimeprime}
\\
X''\supset\Sone\times\Com \times\Com^\ast  &\rightarrow X'_1
\nonumber\\
(u,v_1,v_2)&\mapsto (u\vert v_2\vert,\frac{v_1}{\vert v_2\vert},\frac{v_2}{\vert v_2\vert}),
\label{eq:gluing_embedding_Xprimeprime}
\end{align}
and hence, the composition
\begin{equation}\label{eq:transition_to_second_resolution}
X_\ast\xrightarrow{f^{-1}} X'_\ast\xrightarrow{\eqref{eq:normalize}} X_\ast''\xrightarrow{\eqref{eq:gluing_embedding_Xprimeprime}} X_1,
\end{equation}
is given by the assignment
\[(p,q_1,q_2)\mapsto (q_2^{-1},q_1,p\vert q_2\vert^2(\frac{q_2}{\vert q_2\vert})^{b_1}),\]
where $X_\ast:=\Com\times\Sone\times\Com^\ast\subset X$,
$X_\ast':=\Com\times\Com^\ast\times\Sone\subset X'$,
and  
$X_\ast'':=\Sone \times\Com^\ast\times\Com\subset X''$.
The composition \eqref{eq:transition_to_second_resolution} restricts to the following assignment  
\begin{equation}\label{eq:map_R_Rperp}
(0,q_1,q_2)\mapsto (q_2^{-1},q_1,0)
\end{equation} 
on $0\times \Sone\times \Com^\ast$.

By induction, $g_i$, $i=2,\dots,n-1$
%there exist $Z_i$, $i=0,\dots, n-2$, 
%$Z_{n-2}'$ and $h_i$, $i=1,\dots,n-2$, $h$
satisfying conditions required 
%for $X_i,g_i$ 
in Lemma \ref{lm:chain_charts},
and an embedding $e''$ from $X''-\Fib_0$
to 
\[X_\omega'':=X_1\cup_{g_2} X_2\cup_{g_3} \cdots\cup_{g_{n-1}} X_{n-1}\cup_{g} X_{n-1}',\]
such that the resolution of $\Fib_0$ in $(M',\alpha')$
is obtained by gluing $X_\omega''$ to $X''-\Fib_0$
via $e''$. Let $X_0:= X$,
% $X_i:=Z_{i-1}$, $g_i:=h_{i-1}$, $i>1$,
and $g_1$ be the composition \eqref{eq:transition_to_second_resolution}. Then $e''$ and \eqref{eq:gluing_embedding1_5_dim}
induce a gluing embedding from $U_\Fib-\Fib$ to  
\[X_\omega:=X_0\cup_{g_1} X_1\cup_{g_2}\cdots\cup_{g_{n-1}} X_{n-1}\cup_{g} X_{n-1}'.\]
Furthermore, by \eqref{eq:map_R_Rperp}, we have 
\begin{multline*}
X_\omega\supset R_0\cup R_1\cup_{g_2} \cdots\cup_{g_{n-1}} (R_{n-1}^\perp\cup R_{n-1})\cup_{g} R_{n-1}'\\
= 
R_0\cup_{g_1}(R_1^\perp \cup R_1)\cup_{g_2} \cdots\cup_{g_{n-1}} (R_{n-1}^\perp\cup R_{n-1})\cup_{g} R_{n-1}',
\end{multline*}
and there the claim.
\end{proof}

Note that if $M$ is simply connected, then $M_\omega$ is diffeomorphic to $M\# n (\Sthree\times \Stwo)$.

%check this?
\subsection{Compatibility with resolution in $4$-dimensions}\label{subsec:compatibility}
Given a complex $2$-dimensional orbifold $O$ 
and let $p$ be an isolated singularity of type 
$\frac{1}{r}(1,a)$.
Then based on Reid's model \cite[]{Rei}, 
one can resolve the singularity
by replacing a neighborhood $\ba U_p$
of $p$ with a $4$-manifold given by gluing $n+1$ copies of 
$\Com^2$, denoted by $Y_0,\dots,Y_n$. 
The gluing maps $f_i$ between $Y_i$, $i=0,\dots,n$, 
is given by  
\begin{align*}
f_i:Y_i-(\Com\times 0) &\rightarrow Y_{i+1}\\
(\xi_i,\eta_i)&\mapsto (\eta_i^{-1},\xi_{i}(\eta_i)^{b_i}),
& i=0,\dots,n-1,
\end{align*} 
and the embedding from $\ba U_p-p\simeq \Com^2-\{(0,0)\}$ to $Y_1\cup_{f_1}\cup\cdots\cup_{f_{n-1}}Y_n$ given by 
\begin{align}
\ba U_p-V_0&\rightarrow Y_0\nonumber\\
(v_1,v_2)&\mapsto (v_1^r,v_2v_1^{-a})\nonumber\\
\ba U_p-V_1&\rightarrow Y_1\nonumber\\
(v_1,v_2)&\mapsto (v_2^{-1}v_1^a, v_1^r(v_2{v_1}^{-a})^{b_1})\nonumber\\
\cdots \nonumber\\
\cdots \label{eq:reid_embeddings}
\end{align}
allows us to glue $\ba U_p-\{p\}$ with 
$Y_0\cup_{f_0} Y_1\cup_{f_1}\cdots\cup_{f_{n-1}}Y_n$, 
where $V_j$ is 
$\Com\times 0$, $0\times \Com$ or their union 
depending on the assignments in \eqref{eq:reid_embeddings}, 
and $\frac{r}{a}=[b_1, \dots, b_n]$ is the fraction expression
of $\frac{r}{a}$.

Reid's model can be 
viewed as a reduction where we 
first reduce the singularity to a singularity of 
type $\frac{1}{a}(1,a_1)$ by gluing
$U_p-p$ with $Y\cup_f (Y'/\mathbb{Z}_a)$, where
$Y,Y'$ are two copies of $\Com^2$ 
with $\mathbb{Z}_a$ acting on $Y'$ by 
\[(\zeta_1',\zeta_2')\mapsto (\epsilon^{a_1}\zeta_1',\epsilon\zeta_2'),\quad \epsilon=e^{2\pi i/a}.\]

To find the gluing map $f$, 
we apply the embedding in \eqref{eq:reid_embeddings} to $Y''\simeq Y'$ with
\begin{align*}
Y'&\xrightarrow{\sim} Y''\\
(\zeta_1',\zeta_2')&\mapsto (\zeta_2',\zeta_1'),
\end{align*}
and identify $Y$ with $Y_0$.\footnote{The use of $Y''$ 
comes in handy when later we compare 
it with the resolution of the singular fiber of $\pair$.}
In particular, this implies the equality:
\begin{equation}
\big((\zeta_2')^a, \zeta_1'(\zeta_2')^{-a_1}\big) 
=\big(v_2^{-1}v_1^a, v_1^r(v_2v_1^{-a})^{b_1}\big).
\end{equation}
The gluing map $f$
can then be expressed in terms of coordinates as follows:
\begin{align*}
f:(Y'- 0\times \Com) &\rightarrow Y \\
(\zeta_1',\zeta_2')&\mapsto \big((\zeta_2')^r\zeta_1',(\zeta_2')^{-a}\big).
\end{align*}
Similarly, we can write down the embedding
\[\big(U_p-\{(0,0)\}\big)\rightarrow  Y\cup_f (Y'/\mathbb{Z}_a)\] 
as follows:
\begin{align}
(U_p-0\times \Com) &\rightarrow Y,\nonumber \\
(v_1,v_2)&\mapsto (v_1^r,v_2v_1^{-a}),\label{eq:gluing_embedding_1_dim_4}\\
(U_p-\Com\times 0) &\rightarrow Y', \nonumber\\
(v_1,v_2)&\mapsto (v_2^{\frac{r}{a}},v_1v_2^{-\frac{1}{a}}).\label{eq:gluing_embedding_2_dim_4}
\end{align}
Note the last map involves fraction,
so it is well-defined only after quotient the 
$\mathbb{Z}_a$-action.

Observe that the resolution $O'$ of $O$
of the singular point $p$ is topologically equivalent to
the manifold
\begin{equation}\label{eq:resolution_topological_model}
\big((O-\mathring{D_U})/\mathbb{Z}_r\big)\cup_h \big(D_Y\cup_f( D_Y'/\mathbb{Z}_a)\big),
\end{equation}
where $D_U,D_Y,D_Y'$ are the product of two unit disks in $U_p,Y,Y'$, respectively,
with the gluing map $h$ induced by 
the map 
\[\partial D_U=\Sone\times D^2\cup D^2\times\Sone\rightarrow \Sone\times D^2\cup_f (\Sone\times D^2/\mathbb{Z}_a)\] 
given by the assignments 
\begin{align}
\Sone\times D^2&\rightarrow \Sone\times D^2\subset D_Y\nonumber\\
(v_1,v_2)&\mapsto (v_1^r,v_2v_1^{-a})
%(v_1,v_2)&\mapsto (x^r,yx^{-a})
\label{eq:gluing1_topological_model}\\
D^2\times \Sone&\rightarrow \Sone\times D^2\subset D_Y'\nonumber\\
(v_1,v_2)&\mapsto (v_2^{\frac{r}{a}},v_1v_2^{-\frac{1}{a}})
%(v_1,v_2)&\mapsto (y^{\frac{r}{a}},xy^{-\frac{1}{a}})
\label{eq:gluing2_topological_model}
\end{align}
and $f$ induced by the map
\begin{align}
D_Y'\supset D^2\times \Sone &\rightarrow  D^2\times \Sone\subset D_Y\nonumber\\
(\zeta_1',\zeta_2')&\mapsto (\zeta_1'(\zeta_2')^r,(\zeta_2')^{-a}).\label{eq:gluing3_topological_model}
\end{align} 
 
To see the resolution given in Section \ref{subsec:reduction}
descends to the resolution for complex orbifolds topologically.
We recall the map \eqref{eq:tilXprime_to_Xprime} 
from $\til X'$ to $X'$:
\[
(\til p',\til q_1',\til q_2')\mapsto 
((\til q_2')^{r}\til p',(\til q_2')^{-1}\til q_1',(\til q_2')^{-a}),
\] 
and denote it by $\pi'$. Similarly, 
we have an $\Sone$-equivariant map $\pi$ from 
$\til X=\Com\times \Sone\times \Com$ to $X$ given by 
\begin{equation}\label{eq:tilX_to_X}
(\til p,\til q_1,\til q_2)\mapsto 
((\til q_1)^{r}\til p,\til q_1^{-1}, \til q_2\til q_1^{-a}),
\end{equation}
where the $\Sone$-action acting
on the second factor of $\til X$.
Note $\pi$ is an $\Sone$-equivariant diffeomorphism, 
whereas $\pi'$ is not a diffeomorphism.
These two maps together with the gluing map $f$
in \eqref{eq:gluing_map_XprimeX} give us the composition
$\pi^{-1}\circ f\circ \pi'$ which in terms of the coordinates
of $\til X'$ can be written down as follows:
\begin{equation}\label{eq:gluing_tilXprime_to_tilX}
(\til p',\til q_1', \til q_2')\mapsto \big(\til p'\vert\til q_1'\vert 
(\frac{\til q_1'}{\vert \til q_1'\vert})^r, (\til q_1')^{-1}\til q_2',(\til q_1')^{-a}\vert \til q_1'\vert^{a-1}\big).
\end{equation}

Quotient out the $\Sone$-action, 
we obtain that $\ba X=\til X/\Sone$, $\ba X'=\til X'/\Sone$, and that the gluing map \eqref{eq:gluing_tilXprime_to_tilX} descends to the following map 
\begin{align*}
\ba X'\supset\Com\times\Com^\ast&\rightarrow \Com\times \Com^\ast\subset \ba X\\
(\til p',\til q_1')&\mapsto (\til p'\vert\til q_1'\vert
(\frac{\til q_1'}{\vert \til q_1'\vert})^r, (\til q_1')^{-a}\vert \til q_1\vert^{a-1}).
\end{align*} 
In particular, when restricted to $D^2\times\Sone\subset \ba X'$, 
it yields precisely the assignment in \eqref{eq:gluing3_topological_model}
since now $\vert \til q_1'\vert=1$. 

Now, we want to show 
the gluing map $e$ in \eqref{eq:gluing_embedding} 
descends to 
the gluing maps \eqref{eq:gluing_embedding_1_dim_4} and
\eqref{eq:gluing_embedding_2_dim_4} in $4$-dimensions.  
As with the case of $X,X'$, we consider a
model complex $\Sone$-manifold $\til U_\Fib\simeq \Sone\times\Com^2$ for $U_\Fib$, 
as in \eqref{eq:s1_action_nbhd_singular}:
\begin{align}
\til U_\Fib  &\rightarrow U_\Fib\nonumber\\
(\til u,\til v_1,\til v_2)&\mapsto (\til u^r,\til u^{-1}\til v_1, \til u^{-a} \til v_2).\label{eq:model_neighborhood_U_S} 
\end{align}
Compose the maps \eqref{eq:model_neighborhood_U_S}, 
\eqref{eq:gluing_embedding1_5_dim}, and \eqref{eq:gluing_embedding2_5_dim}, and combine
them with the maps $\pi:\til X\rightarrow X,\pi':\til X'\rightarrow X'$, we obtain the assignments:
\begin{align*}
\til U_\Fib\supset\Sone\times\Com^\ast\times\Com  
& \rightarrow \til X\\
(\til u,\til v_1,\til v_2)&\mapsto (\vert \til v_1\vert(\frac{\til v_1}{\vert \til v_1\vert})^r,\vert \til v_1\vert\til v_1^{-1}\til u, \til v_2 \til v_1^{-a}\vert\til v_1\vert^{a-1})\\
\til U_\Fib\supset\Sone\times\Com\times\Com^\ast 
& \rightarrow \til X'\\
(\til u,\til v_1,\til v_2)&
\mapsto (\vert \til v_2\vert(\frac{\til v_2}{\vert \til v_2\vert})^\frac{r}{a}, \vert \til v_2\vert^{\frac{1}{a}-1}\til v_2^{-\frac{1}{a}}\til v_1, 
\til v_2^{-\frac{1}{a}} \vert\til v_2\vert ^{\frac{1}{a}} \til u),
\end{align*}  
which descend to the following assignments 
after quotient out the $\Sone$-action:
\begin{align}
\Com^\ast\times\Com  
& \rightarrow \ba X\nonumber\\
(\til v_1,\til v_2)&\mapsto (\vert \til v_1\vert(\frac{\til v_1}{\vert \til v_1\vert})^r, \til v_2 \til v_1^{-a}\vert\til v_1\vert^{a-1})\label{eq:quotient_gluing_embedding1}\\ 
\Com\times\Com^\ast 
& \rightarrow \ba X'\nonumber\\
(\til v_1,\til v_2)&\mapsto 
(\vert \til v_2\vert(\frac{\til v_2}{\vert \til v_2\vert})^\frac{r}{a}, \vert \til v_2\vert^{\frac{1}{a}-1}\til v_2^{-\frac{1}{a}}\til v_1).\label{eq:quotient_gluing_embedding2}
\end{align}  
The assignment \eqref{eq:quotient_gluing_embedding1}
restricts to the following on $\Sone\times D^2$: 
\begin{align*}
\Sone \times D^2  
& \rightarrow \Sone\times D^2\subset \ba X\nonumber\\
(\til v_1,\til v_2)&\mapsto ( \til v_1^r, \til v_2 \til v_1^{-a}),
\end{align*} 
while the assignment \eqref{eq:quotient_gluing_embedding2} the following on $D^2\times \Sone$:
\begin{align*}
D^2 \times\Sone  
& \rightarrow \Sone\times D^2\subset \ba X'\nonumber\\
(\til v_1,\til v_2)&\mapsto ( \til v_2^\frac{r}{a},  \til v_1\til v_2^{-\frac{1}{a}}).
\end{align*} 
They are those in \eqref{eq:gluing1_topological_model} and \eqref{eq:gluing2_topological_model}, respectively,
so $M'/\Sone$ is homeomorphic to $O'$, and hence, by induction, 
the third assertion of Theorem \ref{teo:resolution}; this
completes the proof.